\documentclass[a4paper,12pt,final]{amsart}
\usepackage{times,a4wide,mathrsfs,amssymb}

\newcommand{\C}{\mathbb{C}}

\newcommand{\QQ}{\mathbb{Q}}
\newcommand{\NN}{\mathbb{N}}
\newcommand{\PP}{\mathbb{P}}

\newcommand{\OO}{\mathcal O}

\newcommand{\MM}{\mathcal M}

\newcommand{\ima}{\hbox{Im}}
\newcommand{\rom}{\romannumeral}
\newcommand{\alb}{\hbox{Alb}}

\newtheorem{theorem}{Theorem}[section]

\newtheorem{corollary}[theorem]{Corollary}
\newtheorem{proposition}[theorem]{Proposition}
\newtheorem{conjecture}[theorem]{Conjecture}
\newtheorem{remark}[theorem]{Remark}
\newtheorem{definition}[theorem]{Definition}
\newtheorem{convention}{Conventions}

\newtheorem{nonumbering}{Theorem}

\newtheorem{nonumberingt}{Acknowledgements}

\begin{document}
\author[Robert Laterveer]
{Robert Laterveer}

\address{Institut de Recherche Math\'ematique Avanc\'ee,
CNRS -- Universit\'e 
de Strasbourg,\
7 Rue Ren\'e Des\-car\-tes, 67084 Strasbourg CEDEX,
FRANCE.}
\email{robert.laterveer@math.unistra.fr}

\title{Algebraic cycles on surfaces with $p_g=1$ and $q=2$}

\begin{abstract} This note is about an old conjecture of Voisin, which concerns zero--cycles on the self--product of surfaces of geometric genus one. We prove this conjecture for surfaces with $p_g=1$ and $q=2$.
\end{abstract}



\maketitle

\section{Introduction}

For $X$ a smooth projective variety over $\C$, let $A^jX$ denote the Chow groups of codimension $j$ algebraic cycles on $X$ modulo rational equivalence.
For varieties of dimension larger than $1$,
the Chow groups are still incompletely understood. For example, there is the famous conjecture of Bloch:

\begin{conjecture}[Bloch \cite{B}]\label{bloch} Let $S$ be a smooth projective complex surface. The following are equivalent:

\noindent
(\rom1) the Albanese morphism $A^2_{}(S)\to\alb(S)$ is injective;

\noindent
(\rom2) the geometric genus $p_g(S)$ is $0$.
\end{conjecture}

The implication from (\rom1) to (\rom2) is actually a theorem \cite{Mum}, \cite{BS}. The conjectural part is the reverse implication, which is wide open for surfaces of general type (cf. \cite{BKS}, \cite{PW} for cases where conjecture \ref{bloch} is verified).

The next step is to consider surfaces $S$ with $p_g=1$. Here, the Albanese kernel $A^2_{AJ}(S)$ is huge (it is ``infinite--dimensional'', in a certain sense \cite{Mum}). Yet, at least conjecturally, this huge group has controlled behaviour on the self--product:

\begin{conjecture}[Voisin \cite{V9}\footnote{To be precise, conjecture \ref{vois} is not stated verbatim in \cite{V9}. Conjecture \ref{vois} is, however, close in spirit to \cite{V9}; for instance, in case $q(S)=0$ conjecture \ref{vois} coincides with \cite[Section 3 conjecture (*)]{V9}.}]\label{vois} Let $S$ be a smooth projective complex surface. The following are equivalent:

\noindent
(\rom1) 
 For any $a,a^\prime\in A^2_{AJ}(S)$, we have
  \[ a\times a^\prime=a^\prime\times a\ \ \ \hbox{in}\ A^4(S\times S)\ .\]
  
 \noindent
 (\rom2) the geometric genus $p_g(S)$ is $\le 1$.
 \end{conjecture}
 
 Again, the implication from (\rom1) to (\rom2) is actually a theorem (this can readily be proven \`a la Bloch--Srinivas \cite{BS}). The conjectural part is the reverse implication. This has been proven in some special cases \cite{V9}, \cite{moi}, but is still wide open for a general $K3$ surface.
 
 The aim of this modest note is to exhibit some more examples of surfaces verifying conjecture \ref{vois}. The main result states that conjecture \ref{vois} is verified for the surfaces mentioned in the title:
 
 \begin{nonumbering}[=theorem \ref{main}] Let $S$ be a smooth projective surface with $p_g(S)=1$ and $q(S)=2$. Let $a,a^\prime\in A^2_{AJ}(S)$. Then
  \[ a\times a^\prime=a^\prime\times a\ \ \ \hbox{in}\ A^4(S\times S)\ .\] 
  \end{nonumbering}
  
  This result is proven by reduction to the case of abelian surfaces, where the conjecture is known to hold by work of Voisin \cite{Vo}. This reduction step can be made thanks to the theory of finite--dimensional motives developed by Kimura and O'Sullivan \cite{Kim}, \cite{An}. As a corollary to theorem \ref{main}, a certain instance of the generalized Hodge conjecture is verified on $S\times S$ (corollary \ref{GH}).
  
  The last section of this note is about decomposability of zero--cycles on surfaces with $p_g=1$ and $q=2$. Here, we raise some questions which we hope might spawn further research (conjecture \ref{nondecomp} and remark \ref{motiv}).

\vskip0.6cm

\begin{convention} In this note, the word {\sl variety\/} will refer to a reduced irreducible scheme of finite type over $\C$. The word {\sl surface\/} will refer to a smooth projective variety of dimension $2$. As customary, we will write $p_g(S):=\dim H^0(S,\Omega_S^2)$ for the geometric genus, and $q(S):=\dim H^0(S,\Omega^1_S)$ for the irregularity of a surface.

For a variety $X$, we will denote by $A_jX$ the Chow group of $j$--dimensional cycles on $X$; 
for $X$ smooth of dimension $n$ the notations $A_jX$ and $A^{n-j}X$ will be used interchangeably. 

The notations 
$A^j_{hom}(X)$ and $A^j_{AJ}(X)$ will be used to indicate the subgroups of 
homologically, resp. Abel--Jacobi trivial cycles.
The (contravariant) category of Chow motives (i.e., pure motives with respect to rational equivalence as in \cite{Sc}, \cite{MNP}) will be denoted $\MM_{\rm rat}$. The category of pure motives with respect to homological equivalence will be denoted $\MM_{\rm hom}$.
\end{convention}

\section{Preliminaries}

\subsection{Finite--dimensional motives}

We refer to \cite{Kim}, \cite{An}, \cite{Iv}, \cite{J4}, \cite{MNP} for the definition of finite--dimensional motive. 
An essential property of varieties with finite--dimensional motive is embodied by the nilpotence theorem:

\begin{theorem}[Kimura \cite{Kim}]\label{nilp} Let $X$ be a smooth projective variety of dimension $n$ with finite--dimensional motive. Let $\Gamma\in A^n(X\times X)_{\QQ}$ be a correspondence which is numerically trivial. Then there is $N\in\NN$ such that
     \[ \Gamma^{\circ N}=0\ \ \ \ \in A^n(X\times X)_{\QQ}\ .\]
\end{theorem}

 Actually, the nilpotence property (for all powers of $X$) could serve as an alternative definition of finite--dimensional motive, as shown by a result of Jannsen \cite[Corollary 3.9]{J4}.
Conjecturally, any variety has finite--dimensional motive \cite{Kim}. We are still far from knowing this, but at least there are quite a few non--trivial examples:
 
\begin{remark} 
The following varieties have finite--dimensional motive: abelian varieties, varieties dominated by products of curves \cite{Kim}, $K3$ surfaces with Picard number $19$ or $20$ \cite{P}, surfaces not of general type with vanishing geometric genus \cite[Theorem 2.11]{GP}, Godeaux surfaces \cite{GP}, certain surfaces of general type with $p_g=0$ \cite{V8}, \cite{BF},\cite{PW}, Hilbert schemes of surfaces known to have finite--dimensional motive \cite{CM}, generalized Kummer varieties \cite[Remark 2.9(\rom2)]{Xu},
 3--folds with nef tangent bundle \cite{Iy} (an alternative proof is given in \cite[Example 3.16]{V3}), 4--folds with nef tangent bundle \cite{Iy2}, log--homogeneous varieties in the sense of \cite{Br} (this follows from \cite[Theorem 4.4]{Iy2}), certain 3--folds of general type \cite[Section 8]{V5}, varieties of dimension $\le 3$ rationally dominated by products of curves \cite[Example 3.15]{V3}, varieties $X$ with $A^i_{AJ}(X)_{}=0$ for all $i$ \cite[Theorem 4]{V2}, products of varieties with finite--dimensional motive \cite{Kim}.
\end{remark}

\begin{remark}
It is a (somewhat embarrassing) fact that all examples known so far of finite-dimensional motives happen to be in the tensor subcategory generated by Chow motives of curves (i.e., they are ``motives of abelian type'' in the sense of \cite{V3}). 
That is, the finite--dimensionality conjecture is still unknown for any motive {\em not\/} generated by curves (on the other hand, there exist many motives not generated by curves, cf. \cite[7.6]{D}).
\end{remark}

\subsection{The transcendental motive}

\begin{theorem}[Kahn--Murre--Pedrini \cite{KMP}]\label{t2} Let $S$ be a surface. There exists a decomposition
  \[ h_2(S)= t_2(S)\oplus h_2^{alg}(S)\ \in \MM_{\rm rat}\ ,\]
  such that
  \[  H^\ast(t_2(S),\QQ)= H^2_{tr}(S)\ ,\ \ H^\ast(h_2^{alg}(S),\QQ)=NS(S)_{\QQ}\ \]
  (here $H^2_{tr}(S)$ is defined as the orthogonal complement of the N\'eron--severi group $NS(S)_{\QQ}$ in $H^2(S,\QQ)$),
  and
   \[ A^\ast(t_2(S))_{\QQ}=A^2_{AJ}(S)_{\QQ}\ .\]
   (The motive $t_2(S)$ is called the {\em transcendental motive\/}.)
   \end{theorem}

\subsection{Surfaces isogenous to a product}

\begin{definition}[] Let $S$ be a smooth projective surface. We say that $S$ is {\em isogenous to a product\/} if $S$ admits an unramified finite covering which is biholomorphic to a product of two curves, both of genus at least $1$.
\end{definition}

\begin{remark} Surfaces isogenous to a product have been intensively studied by Catanese and his school, spawning a wealth of interesting concrete examples, cf. \cite{Cat}, \cite{BCP} and the (many) references given there.
\end{remark}

\begin{definition}[\cite{Cat}] Let $S$ be a smooth projective surface. We say that $S$ is {\em generalized hyperelliptic\/} if $S$ can be written
  \[ S=(C\times F)/G\ ,\]
  where $C,F$ are curves of genus at least $1$, and $G$ is a finite subgroup of the product of automorphism groups of $C$ and $F$, and $C\to C/G$ is unramified and $F/G\cong\PP^1$.
\end{definition}

\begin{remark} Generalized hyperelliptic surfaces are studied in \cite{Cat} and \cite{Zuc}.
\end{remark}

\section{Main theorem}

\begin{theorem}\label{main} Let $S$ be a surface with $p_g=1$ and $q=2$. Let $a,a^\prime\in A^2_{AJ}(S)$. Then
  \[ a\times a^\prime=a^\prime\times a\ \ \ \hbox{in}\ A^4(S\times S)\ .\]
  (Here $a\times a^\prime$ is a short--hand for the cycle class $(p_1)^\ast(a)\cdot(p_2)^\ast(a^\prime)\in A^4(S\times S)$, where $p_1, p_2$ denote projection on the first, resp. second factor.)
  \end{theorem}

\begin{proof} First, Rojtman's theorem \cite{Ro} says there is no torsion in $A^4_{AJ}(S\times S)$, and so it suffices to prove the theorem for Chow groups with $\QQ$--coefficients.
Next, since $p_g$ and $q$ and $A^2_{AJ}$ are birational invariants (of smooth surfaces), we may suppose $S$ is minimal. 
We now use the following structural results:

\begin{proposition}[Ram\'on Mar\'\i \cite{RM}]\label{RM1} Let $S$ be a minimal surface with $p_g(S)=1$ and $q=2$, and suppose $S$ is not abelian. Then $S$ is isogenous to a 
product (in particular, $S$ has finite--dimensional motive).

  \end{proposition}
  
  \begin{proof} \cite[Proposition 4.1]{RM}.
  \end{proof}
  
  \begin{proposition}[Ram\'on Mar\'\i \cite{RM}]\label{RM2} Let $S$ be as in proposition \ref{RM1}. One of the following cases occurs:
  
  \noindent
  (\rom1) The Albanese map induces an isomorphism of motives
    \[ h(S)\cong h(\alb(S))\ \ \ \hbox{in}\ \MM_{\rm rat}\ .\]
    
  \noindent
  (\rom2) The Albanese map sends $S$ to a curve, $S$ is generalized hyperelliptic, and there exists an isomorphism of motives  
     \[ t_2(S)\cong t_2(A)\ \ \ \hbox{in}\ \MM_{\rm rat}\ ,\]
     where $A$ is an abelian surface.
  \end{proposition}
  
\begin{proof} \cite[Proposition 4.4]{RM} proves that if the Albanese map is surjective, then it induces an isomorphism of homological motives
  \[ h(S)\cong h(\alb(S))\ \ \ \hbox{in}\ \MM_{\rm hom}\ .\]
  Since both sides are finite--dimensional motives, this implies an isomorphism of Chow motives
  \[ h(S)\cong h(\alb(S))\ \ \ \hbox{in}\ \MM_{\rm rat}\ .\]
 
 Suppose now the Albanese map is not surjective. Then \cite[Proposition 4.4]{RM} implies that $S$ can be written
   \[ S=(C\times E)/G\ ,\]
   where $C$ is a curve and $E$ is an elliptic curve and $E/G\cong\PP^1$. It is also shown in loc. cit. that $C/G$ is the image of the Albanese map (which is thus a holomorphic fibre bundle with fibre $E$ and base $C/G$). It follows from \cite[Theorem F]{Cat} or, alternatively, from \cite[Proposition 2.3]{Zuc}, that $S$ is generalized hyperelliptic.
  Moreover, \cite[Theorem 4.10]{RM}
 proves the existence of an isomorphism of homological motives
 \[  t_2(S)\cong t_2(A)\ \ \ \hbox{in}\ \MM_{\rm hom}\ ,\]
     where $A$ is an abelian surface. Again, finite--dimensionality allows to upgrade to an isomorphism of Chow motives.
\end{proof}     
     
We can now wrap up the proof of theorem \ref{main}. It follows from proposition \ref{RM2} that there exists an abelian surface $A$ and a correspondence $\Gamma\in A^2(S\times A)_{\QQ}$ inducing an isomorphism
  \[ \Gamma_\ast\colon\ \ A^2_{AJ}(S)_{\QQ}\ \xrightarrow{\cong}\ A^2_{AJ}(A)_{\QQ}\ .\]
There is a commutative diagram
  \[ \begin{array}[c]{ccc}
          A^2_{AJ}(S)_{\QQ}\otimes A^2_{AJ}(S)_{\QQ} &\xrightarrow{\nu}& A^4(S\times S)_{\QQ}\\
           \downarrow{(\Gamma_\ast,\Gamma_\ast)} &&\downarrow{(\Gamma\times\Gamma)_\ast}\\
           A^2_{AJ}(A)_{\QQ}\otimes A^2_{AJ}(A)_{\QQ} &\xrightarrow{}& A^4(A\times A)_{\QQ}\\
           \end{array}\]     
The restriction 
  \[ (\Gamma\times\Gamma)_\ast\colon\ \ \ima\nu\ \to\ A^4(A\times A)_{\QQ} \]
  is a split injection (a left inverse is given by $(\Psi\times\Psi)_\ast$, where $\Psi\in A^2(A\times S)_{\QQ}$ is a correspondence inverse to $\Gamma$). It thus suffices to prove conjecture \ref{vois} is true for $A$, which is well--known:
  
\begin{proposition}[Voisin \cite{Vo}] Let $A$ be an abelian variety of dimension $g$. Let $a,a^\prime\in A^g_{(g)}(A)_{\QQ}$, where $A^\ast_{(\ast)}$ denotes the Beauville filtration \cite{Beau}. Then
  \[a\times a^\prime=(-1)^g a^\prime\times a\ \ \ \hbox{in}\ A^{2g}(A\times A)_{\QQ}\ .\]
  \end{proposition}
  
  This is \cite[Example 4.40]{Vo}. Noting that for an abelian surface $A$, we have an inclusion $A^2_{AJ}(A)_{\QQ}\subset A^2_{(2)}(A)_{\QQ}$, this ends the proof.  
\end{proof}

As a corollary to theorem \ref{main}, a certain case of the generalized Hodge conjecture is verified:

\begin{corollary}\label{GH} Let $S$ be a surface with $p_g=1$ and $q=2$. Then the Hodge substructure 
   \[ \wedge^2:=\ima\Bigl(\wedge^2 H^2(S,\QQ)\ \to\ H^4(S\times S,\QQ)\Bigr)\]
 is supported on a divisor $D\subset S\times S$.
 \end{corollary}
 
 \begin{proof} As noted in \cite[Corollary 3.5.1]{V9}, this kind of result follows from the truth of conjecture \ref{vois} by applying the Bloch--Srinivas argument. The idea is to consider the correspondence
    \[ p:= (\Delta_{S\times S}-\Gamma_\iota)\circ(\pi_2\times\pi_2)\ \ \in A^4(S^4)_{\QQ} \]
    (where $\pi_2$ denotes the middle K\"unneth component of $S$, and $\Gamma_\iota$ denotes the graph of the involution $\iota$ of $S\times S$ exchanging the two factors). This correspondence $p$ is a projector on $\wedge^2$.
   For an appropriate choice of $\pi_2$, there is a decomposition
   \[ \pi_2=\pi_2^{tr}+\pi_2^{alg}\ \ \ \hbox{in}\ A^2(S\times S)_{\QQ}\ ,\]
   where $\pi_2^{tr}$ defines the transcendental motive $t_2(S)$ of \cite{MNP}. 
   This induces a decomposition of the correspondence $p$:
   \[ p= (\Delta_{S\times S}-\Gamma_\iota)\circ(\pi_2^{tr}\times\pi_2^{tr}+\pi_2^{tr}\times\pi_2^{alg}+\pi_2^{alg}\times\pi_2^{tr}+\pi_2^{alg}\times\pi_2^{alg})\ \ \in A^4(S^4)_{\QQ}\ . \]
   Clearly, the $3$ summands containing $\pi_2^{alg}$ project $H^4(S\times S)$ to something supported on a divisor (indeed, $\pi_2^{alg}\subset S\times S$ is supported on a 
   product of divisors, by construction \cite{MNP}). It only remains to analyze the first summand
   \[ p^{tr}:=(\Delta_{S\times S}-\Gamma_\iota)\circ(\pi_2^{tr}\times\pi_2^{tr})\ \ \ \in A^4(S^4)_{\QQ}\ \]
   (the correspondence $p^{tr}$ is a projector on $\wedge^2 H^2_{tr}(S)$).   
       
   It follows from theorem \ref{main} that   
      \[ (p^{tr})_\ast A_0(S\times S)_{\QQ}=0\ .\]
  Using the Bloch--Srinivas method \cite{BS}, this implies $p^{tr}$ is rationally equivalent to a cycle supported on $S\times S\times D$, where $D\subset S\times S$ is a divisor. It follows that
  \[   \wedge^2 H^2_{tr}(S)= (p^{tr})_\ast H^4(S\times S) \]
  is supported on $D$.
 \end{proof}

\section{Decomposable cycles}

Let $S$ be a surface. The study of the {\em decomposable cycles\/}, i.e. the image of the intersection map
  \[ A^1(S)_{\QQ}\otimes A^1(S)_{\QQ}\ \to\ A^2(S)_{\QQ} \]
  is an active research topic. For $K3$ surfaces, the group of decomposable cycles has dimension $1$ \cite{BV}.\footnote{There are related results in higher dimension: Let $X$ be a Calabi--Yau complete intersection, and $\dim X=n$. Then 
     \[ \ima\Bigl( A^j(X)_{\QQ}\otimes A^{n-j}(X)_{\QQ}\to A^n(X)_{\QQ}\Bigr) \]
  has dimension $1$, for any $0<j<n$ \cite{V13}, \cite{LFu}.}
  
 For high degree surfaces in $\PP^3$, the dimension of the group of decomposable cycles can become arbitrarily large \cite{OG}.
  For abelian surfaces, it is well--known that {\em all\/} $0$--cycles are decomposable \cite{Bl2}. The same holds for the Fano surface of lines on a cubic threefold \cite[Example 1.7]{B}.
  
  Surfaces with $p_g=1$ and $q=2$ split in two cases: those with surjective Albanese map, and those where the image of the Albanese map is a curve. The first case behaves like abelian varieties (proposition \ref{alb2}). In the second case, there are non--decomposable zero--cycles (proposition \ref{alb1}). Conjecturally, a stronger statement holds (conjecture \ref{nondecomp}).
  
  \begin{proposition}\label{alb2} Let $S$ be a surface with $p_g(S)=1$ and $q(S)=2$. Assume the image of the Albanese map has dimension $2$. Then the intersection map
  \[ A^1(S)_{\QQ}\otimes A^1(S)_{\QQ}\ \to\ A^2(S)_{\QQ} \]
  is surjective.
  \end{proposition}
  
  \begin{proof} We have seen (proposition \ref{RM2}) that in this case the Albanese map $\alpha$ induces an isomorphism of motives
    \[ \Gamma_\alpha\colon\ \ h(S)\ \xrightarrow{\cong}\ h(A)\ \ \ \hbox{in}\ \MM_{\rm hom}\ \]
    (where $A:=\hbox{Alb}(S)$).
    Since $\alpha$ is finite and surjective, the composition $\alpha_\ast\alpha^\ast\colon H^\ast(A)\to H^\ast(A)$ is a multiple of the identity. This implies $\alpha^\ast$ is an isomorphism on cohomology, and so there is an isomorphism of motives
    \[ {}^t \Gamma_\alpha\colon\ \ h(A)\ \xrightarrow{\cong}\ h(S)\ \ \ \hbox{in}\ \MM_{\rm hom}\ .\]
    By finite--dimensionality, we get an isomorphism of Chow motives
     \[ {}^t \Gamma_\alpha\colon\ \ h(A)\ \xrightarrow{\cong}\ h(S)\ \ \ \hbox{in}\ \MM_{\rm rat}\ .\]
  This implies an isomorphism of rings  
      \[ \alpha^\ast\colon\ \ A^\ast(A)_{\QQ}\ \xrightarrow{\cong}\ A^\ast(S)_{\QQ}\ .\]
    The proposition now follows, since zero--cycles on abelian varieties are decomposable \cite{Bl2}.
    \end{proof}

  \begin{proposition}\label{alb1} Let $S$ be a surface with $p_g(S)=1$ and $q(S)=2$. Assume the image of the Albanese map has dimension $<2$. Then the intersection map
  \[ A^1_{hom}(S)_{\QQ}\otimes A^1_{hom}(S)_{\QQ}\ \to\ A^2_{AJ}(S)_{\QQ} \]
  is not surjective.
    \end{proposition}
  
  \begin{proof} In this case, $S$ is generalized hyperelliptic, more precisely $S$ is of the form
       \[ S=(C\times E)/G\ ,\]
      where $E$ is an elliptic curve and $E/G$ is a rational curve \cite[Proof of Proposition 4.4]{RM}. It follows that
        \[ H^1(S,\QQ)= H^1(C\times E,\QQ)^G= H^1(C,\QQ)^G\oplus H^1(E,\QQ)^G = H^1(C,\QQ)^G\ .\]   
  This implies that the cup product map
   \[  H^1(S,\OO_S)\otimes H^1(S,\OO_S)\ \to\ H^2(S,\OO_S)\ \]
   is zero (indeed, it factors over 
     \[ H^1(C/G,\OO_{C/G})\otimes H^1(C/G,\OO_{C/G})\ \to\ H^2(C/G,\OO_{C/G})\ ,\]
     which clearly must be zero.)  

         Suppose now the intersection map
     \[ A^1_{hom}(S)_{\QQ}\otimes A^1_{hom}(S)_{\QQ}\ \to\ A^2_{AJ}(S)_{\QQ} \]
     were surjective. Then it follows from \cite{ESV} that also     
     \[  H^1(S,\OO_S)\otimes H^1(S,\OO_S)\ \to\ H^2(S,\OO_S)\ \]     
     is surjective. This is a contradiction.
     \end{proof}

Conjecturally, something stronger is true:

\begin{conjecture}\label{nondecomp} Let $S$ be a surface with $p_g(S)=1$ and $q(S)=2$. Assume the image of the Albanese map has dimension $<2$. Then the intersection map
  \[ A^1_{hom}(S)_{\QQ}\otimes A^1_{hom}(S)_{\QQ}\ \to\ A^2_{AJ}(S)_{\QQ} \]
  is zero.
  \end{conjecture}

\begin{remark}\label{motiv} The ``motivation'' underlying conjecture \ref{nondecomp} is as follows. The map
  \[ A^1_{hom}(S)_{\QQ}\otimes A^1_{hom}(S)_{\QQ}\ \to\ A^2_{AJ}(S)_{\QQ} \]
should be somehow related to the cup product map
  \[  H^1(S,\OO_S)\otimes H^1(S,\OO_S)\ \to\ H^2(S,\OO_S)\ .\]
  But we have seen this cup product map is zero.
  
  I can prove conjecture \ref{nondecomp} is true if there exists a Chow motive $dc(S)\in\MM_{\rm rat}$ responsible for the decomposable cycles, i.e. such that
    \[ A^\ast(dc(S))=\ima\bigl( A^1(S)_{\QQ}\otimes A^1(S)_{\QQ}\ \to\ A^2(S)_{\QQ}\bigr)\ . \]
  Can one construct such a motive $dc(S)$ ? This is likely to be related to the concept of ``multiplicative Chow--K\"unneth decomposition'' of Shen--Vial \cite{SV}, \cite{V6}.
  \end{remark}  
  
 \begin{remark} Let $S$ be a surface with finite--dimensional motive, and suppose that the cup product map
      \[  cp\colon\ \ H^1(S,\OO_S)\otimes H^1(S,\OO_S)\ \to\ H^2(S,\OO_S)\ \]  
      is surjective. Then it follows from \cite[Theorem 3]{moimult} that also the intersection product
      \[  ip\colon\ \ A^1_{hom}(S)_{\QQ}\otimes A^1_{hom}(S)_{\QQ}\ \to\ A^2_{AJ}(S)_{\QQ} \]
      is surjective (i.e., all Abel--Jacobi trivial zero--cycles are decomposable). 
        However, the conjectural implication discussed in remark \ref{motiv} (i.e., if $cp$ is zero then also $ip$ is zero) appears to be more difficult to prove.
      \end{remark}

\vskip1cm
\begin{nonumberingt} The ideas developed in this note grew into being during the Strasbourg 2014---2015 groupe de travail based on the monograph \cite{Vo}. Thanks to all the participants of this groupe de travail for the pleasant and stimulating atmosphere. I am grateful to the referee and to the editor for helpful comments.
Many thanks to Yasuyo, Kai and Len for accepting to be honorary members of the Schiltigheim Math Research Institute.
\end{nonumberingt}


\vskip1cm
\begin{thebibliography}{dlPG99}






\bibitem{An} Y. Andr\'e, Motifs de dimension finie (d'apr\`es S.-I. Kimura, P. O'Sullivan,...), S\'eminaire Bourbaki 2003/2004, Ast\'erisque 299 Exp. No. 929, viii, 115---145,




\bibitem{BCP} I. Bauer, F. Catanese and R. Pignatelli, Complex surfaces of general type: Some recent progress, in: Global aspects of complex geometry (F. Catanese et alii, eds.), Springer--Verlag Berlin Heidelberg 2006,

\bibitem{BF} I. Bauer and D. Frapporti, Bloch's conjecture for generalized Burniat type surfaces with $p_g=0$, Rend. Circ. Mat. Palermo 64 (2015), 27---42,


\bibitem{Beau} A. Beauville, Sur l'anneau de Chow d'une vari\'et\'e ab\'elienne, Math. Ann. 273 (1986), 647---651,


\bibitem{BV} A. Beauville and C. Voisin, On the Chow ring of a $K3$ surface, J. Alg. Geom. 13 (2004), 417---426,

\bibitem{Bl2} S. Bloch, Some elementary theorems about algebraic cycles on abelian varieties, Invent. Math. 37 (1976), 215---228,

\bibitem{B} S. Bloch, Lectures on algebraic cycles, Duke Univ. Press Durham 1980,

\bibitem{BKS} S. Bloch, Kas and A. Lieberman, Zero cycles on surfaces with $p_g=0$, Comp. Math. 33 no. 2 (1976), 135---145,


\bibitem{BS} S. Bloch and V. Srinivas, Remarks on correspondences and algebraic cycles, American Journal of Mathematics Vol. 105, No 5 (1983), 1235---1253,



\bibitem{Br} M. Brion, Log homogeneous varieties, in: Actas del XVI Coloquio Latinoamericano de Algebra, 
Revista Matem\'atica Iberoamericana, Madrid 2007,
arXiv: math/0609669,


\bibitem{CM} 
M. de Cataldo and L. Migliorini, The Chow groups and the motive of the Hilbert scheme of points on a
surface, Journal of Algebra 251 no. 2 (2002), 824---848,


\bibitem{Cat} F. Catanese, Fibred surfaces, varieties isogenous to a product and related moduli spaces, Amer. J. Math. 122 no. 1 (2000), 1---44,







\bibitem{D} P. Deligne, La conjecture de Weil pour les surfaces $K3$, Invent. Math. 15 (1972), 206---226,




\bibitem{ESV} H. Esnault, V. Srinivas and E. Viehweg, Decomposability of Chow groups implies decomposability of cohomology, in :``Journ\'ees de G\'eom\'etrie Alg\'ebrique d'Orsay, Juillet 1992'', Ast\'erisque 218 (1993), 227---242, 


\bibitem{LFu} L. Fu, Decomposition of small diagonals and Chow rings of hypersurfaces and Calabi--Yau complete intersections, Advances in Mathematics (2013), 894---924,














\bibitem{GP} V. Guletski\u{\i} and C. Pedrini, The Chow motive of the Godeaux surface, in:
Algebraic Geometry, a volume in memory of Paolo Francia (M.C. Beltrametti et alii, editors),
Walter de Gruyter, Berlin New York, 2002,






\bibitem{Iv} F. Ivorra, Finite dimensional motives and applications (following S.-I. Kimura, P. O'Sullivan and others), in:
        Autour des motifs, Asian-French summer school on algebraic geometry and number theory,
       Volume III, Panoramas et synth\`eses, Soci\'et\'e math\'ematique de France 2011,
    
\bibitem{Iy} J. Iyer, Murre's conjectures and explicit Chow--K\"unneth projectors for varieties with a nef tangent bundle, Transactions of the Amer. Math. Soc. 361 (2008), 1667---1681,

\bibitem{Iy2} J. Iyer, Absolute Chow--K\"unneth decomposition for rational homogeneous bundles and for log homogeneous varieties, Michigan Math. Journal
 Vol.60, 1 (2011), 79---91,
          




\bibitem{J4} U. Jannsen, On finite--dimensional motives and Murre's conjecture, in: Algebraic cycles and motives (J. Nagel and C. Peters, eds.), Cambridge University Press, Cambridge 2007,

\bibitem{KMP} B. Kahn, J. Murre and C. Pedrini, On the transcendental part of the motive of a surface, in: Algebraic cycles and motives (J. Nagel and C. Peters, eds.), Cambridge University Press, Cambridge 2007,


\bibitem{Kim} S. Kimura, Chow groups are finite dimensional, in some sense,
Math. Ann. 331 (2005), 173---201,









\bibitem{moi} R. Laterveer, Some results on a conjecture of Voisin for surfaces of geometric genus one, to appear in Boll. Unione Mat. Italiana, 


\bibitem{moimult} R. Laterveer, On a multiplicative version of Bloch's conjecture, to appear in Beitr\"age zum Algebra und Geometrie,












\bibitem{Mum} D. Mumford, Rational equivalence of $0$--cycles on surfaces, J. Math. Kyoto Univ. Vol. 9 No 2 (1969), 195---204,


\bibitem{MNP} J. Murre, J. Nagel and C. Peters, Lectures on the theory of pure motives, Amer. Math. Soc. University Lecture Series 61, Providence 2013,



\bibitem{OG} K. O'Grady, Decomposable cycles and Noether--Lefschetz loci, arXiv:1502.07897v2,






\bibitem{P} C. Pedrini, On the finite dimensionality of a $K3$ surface, Manuscripta Mathematica 138 (2012), 59---72,



\bibitem{PW} C. Pedrini and C. Weibel, Some surfaces of general type for which Bloch's conjecture holds, to appear in: Period Domains, Algebraic Cycles, and Arithmetic, Cambridge Univ. Press, 2015,



\bibitem{RM} J. Ram\'on Mar\'\i, On the Hodge conjecture for products of certain surfaces, Collect. Math. 59, 1 (2008), 1---26,



\bibitem{Ro} A.A. Rojtman, The torsion of the group of 0--cycles modulo rational equivalence, Annals of Mathematics 111 (1980), 553---569,



\bibitem{Sc} T. Scholl, Classical motives, in: Motives (U. Jannsen et alii, eds.), Proceedings of Symposia in Pure Mathematics Vol. 55 (1994), Part 1, 



\bibitem{SV} M. Shen and C. Vial, The Fourier transform for certain hyperK\"ahler fourfolds, Memoirs of the AMS 240 (2016), no.1139,






\bibitem{V2} C. Vial, Projectors on the intermediate algebraic Jacobians, New York J. Math. 19 (2013), 793---822,

\bibitem{V3} C. Vial, Remarks on motives of abelian type, to appear in Tohoku Math. J.,

\bibitem{V4} C. Vial, Niveau and coniveau filtrations on cohomology groups and Chow groups, Proceedings of the LMS 106(2) (2013), 410---444,

\bibitem{V5} C. Vial, Chow--K\"unneth decomposition for $3$-- and $4$--folds fibred by varieties with trivial Chow group of zero--cycles, J. Alg. Geom. 24 (2015), 51---80,

\bibitem{V6} C. Vial, On the motive of some hyperk\"ahler varieties, J. fur Reine u. Angew. Math.,


\bibitem{V9} C. Voisin, Remarks on zero--cycles of self--products of varieties, in: Moduli of vector bundles, Proceedings of the Taniguchi Congress  (M. Maruyama,  ed.), Marcel Dekker New York Basel Hong Kong 1994,




\bibitem{V13} C. Voisin, Chow rings and decomposition theorems for $K3$ surfaces and Calabi--Yau hypersurfaces, Geom. Topol.
16 (2012), 433---473,



\bibitem{V8} C. Voisin, Bloch's conjecture for Catanese and Barlow surfaces, J. Differential Geometry 97 (2014), 149---175,

\bibitem{Vo} C. Voisin, Chow Rings, Decomposition of the Diagonal, and the Topology of Families, Princeton University Press, Princeton and Oxford, 2014,


\bibitem{Xu} Z. Xu, Algebraic cycles on a generalized Kummer variety, arXiv:1506.04297v1,

\bibitem{Zuc} F. Zucconi, Generalized hyperelliptic surfaces, Transactions of the Amer. Math. Soc. 355 No. 10 (2003), 4045---4059.

\end{thebibliography}
\end{document}